\documentclass[12pt]{amsart}
\usepackage[all]{xy}
\usepackage{amsthm}
\usepackage{amsmath}
\usepackage{amssymb}
\usepackage{color}
\usepackage{float}
\usepackage{graphicx}
\usepackage[english]{babel}
\usepackage{verbatim}

\theoremstyle{plain}

\def\Ab{{\mathbb A}}
\def\kf{{\mathbf k}}

\def\Zb{{\mathbb Z}}

\def\Ac{{\mathcal A}}

\def\Dc{{\mathcal D}}
\def\Fc{{\mathcal F}}

\def\Lc{{\mathcal L}}
\def\Oc{{\mathcal O}}
\def\Uc{{\mathcal U}}
\def\Vc{{\mathcal V}}

\def\mf{{\mathfrak m}}
\def\pf{{\mathfrak p}}
\def\Ds{{\Dc}}
\def\C~{\widetilde{C}}
\def\AV{\Ac\Vc}
\def\AU{\Ac \otimes U}
\def\smash{\Ac\# U(\Vc)}
\def\Lp{\Lc_+}
\def\End{{\rm End}}
\newcommand{\dx}[1]{{\frac{\partial}{\partial x_{#1}}}}
\newcommand{\pd}[2]{{\frac{\partial #1}{\partial #2}}}
\def\bF{{\overline \Fc}}
\def\bro{{\overline \rho}}
\def\id{{\rm id}}

\newtheorem{theorem}{Theorem}[section]
\newtheorem{definition}[theorem]{Definition}
\newtheorem{lemma}[theorem]{Lemma}

\newtheorem{corollary}[theorem]{Corollary}
\newtheorem{solution*}{Solution}

\newtheorem{proposition}[theorem]{Proposition}
\newtheorem{remark}[theorem]{Remark}

\numberwithin{equation}{section}

\newcommand{\Spec}{\operatorname{Spec}}

\newcommand{\Der}{\operatorname{Der}} 

\newcommand\dd[1]{\frac{\partial}{\partial #1}}
\newcommand\eps{\varepsilon}

\DeclareMathOperator{\Hom}{\mathscr{H}\text{\kern -3pt {\calligra\large om}}\,}

\DeclareMathOperator{\Tor}{Tor}
\DeclareMathOperator{\depth}{depth}

\title[$\AV$ modules]
{$\AV$ modules of finite type on affine space}
\author{Yuly Billig}
\address{School of Mathematics and Statistics, Carleton University, Ottawa, Canada}
\email{billig@math.carleton.ca}
\author{Colin Ingalls}
\email{cingalls@math.carleton.ca}
\author{Amir Nasr}
\email{amir.nasr@carleton.ca}

\date{\today}

\makeatletter
\newcommand{\leqnomode}{\tagsleft@true}
\newcommand{\reqnomode}{\tagsleft@false}
\makeatother

\begin{document}
\begin{abstract}
{We study the category of modules admitting compatible actions of the Lie algebra $\Vc$ of vector fields on an affine space and the algebra $\Ac$ of polynomial functions. We show that modules in this category which are finitely generated over $\Ac$, are free. We also show that
this pair of compatible actions is equivalent to commuting actions of the algebra of differential operators and the Lie algebra of vector fields vanishing at the origin. This allows us to construct explicit realizations of such modules as gauge modules.}
\end{abstract}
\maketitle

\begin{section}{Introduction}
The study of Lie algebras of vector fields originated in the works of the founders of Lie theory -- Sophus Lie and \'Elie Cartan~\cite{Lie1},~\cite{Lie2},~\cite{C}.
Lie and Cartan introduced four families of simple Lie algebras of vector fields on an affine space: $W_n$, $S_n$, $H_n$ and $K_n$.
These were the first examples of simple infinite-dimensional Lie algebras.

 David Jordan~\cite{J1},~\cite{J2},  proved that Lie algebra $\Vc$ of vector fields on an arbitrary smooth irreducible affine variety $X$ is simple.
This yields a vast class of simple infinite-dimensional Lie algebras. These Lie algebras have a very different structure from simple finite-dimensional 
Lie algebras over an algebraically closed field of characteristic zero. It was shown in~\cite{BF} that Lie algebras of vector fields may have no
non-zero semisimple or nilpotent elements. For this reason, the standard machinery of roots and weights is not applicable for this class of 
simple Lie algebras.

 A systematic study of representation theory of Lie algebras of vector fields on affine varieties was initiated in~\cite{BFN},~\cite{BN},~\cite{BNZ}.
It is hopeless to expect a classification of general representations of such algebras. We have to impose additional constraints on representations 
that will make development of a structure theory possible, yet still include important and interesting families of representations.

A constraint that we impose is the existence of a compatible action between  the commutative algebra $\Ac$ of functions and the Lie algebra of vector fields $\Vc$ on a variety $X$. We wish to study
$\Ac\Vc$ modules, where compatibility condition on a module $M$ is the Leibniz rule:
$$\eta (f m) = f (\eta m) + \eta(f) m,$$
for $\eta \in \Vc, f \in \Ac$ and $m \in M$.

Viewing $\Ac$ as a module for the Hopf algebra $U(\Vc)$, we can form a smash product $\Ac \# U(\Vc)$~\cite{M}. Then $\AV$ modules are precisely
the modules for the associative algebra $\Ac \# U(\Vc)$.

One family of $\AV$ modules is given by $\Dc$ modules, where $\Dc$ is the associative algebra of differential operators of an arbitrary order on $X$.
Since we have a surjective algebra homomorphism $\smash \rightarrow \Dc$, every $\Dc$ module is automatically an $\AV$ module. An important class of $\Dc$ modules is given
by spaces of sections of vector bundles with a connection on $X$. In this example the covariant derivative $\nabla$ is a $\Vc$-action that satisfies an additional 
property
\begin{equation}
\label{nabla}
\nabla_{f\eta} = f \nabla_\eta, \ \ \ f\in\Ac, \eta \in \Vc.
\end{equation}

In fact, any $\AV$-action satisfying (\ref{nabla}), extends to the action of $\Dc$~\cite{Bav},~\cite{G}.
There are also natural examples of $\AV$ modules that do not satisfy (\ref{nabla}). In particular, the adjoint representation of $\Vc$
is an $\AV$ module in which the $\Vc$-action can not be interpreted as a covariant derivative. In a way, studying $\AV$ modules amounts to finding
structures on vector bundles that give rise to $\Vc$-action on the space of sections, generalizing the concept of a connection.

 This paper has two main results. We prove that when $X= \Ab^n$ is an affine space, every $\AV$ module of finite type, i.e., finitely generated 
over $\Ac$, is maximal Cohen-Macaulay, and hence free as an $\Ac$ module. 
We also show that in the case of an arbitrary smooth irreducible affine variety, $\AV$ modules of a finite type have no $\Ac$-torsion.

 The second result we prove is a remarkable isomorphism of associative algebras in case when $X$ is an affine space:
$$ \smash \cong \Dc \otimes U(\Lc_+),$$
where $\Lp$ is the subalgebra of vector fields on $\Ab^n$ that vanish at the origin.
Note that even though $\Vc$ embeds into $\Dc,$ and $\Lp$ embeds into $\Vc$, the isomorphism maps, restricted to $\Vc$ and $\Lp$ respectively, are
highly non-trivial, and are not the naive embedding maps.

Combining these two results we obtain a description of a finite type $\AV$ modules over an affine space -- these are free $\Ac = \kf [x_1, \ldots, x_n]$ modules of a finite rank $\Ac \otimes U$ with 
commuting actions of $\Dc$ and $\Lp$ where $U$ is a finite dimensional $\kf$ vector space.  The action of $\Lp$ must be $\Ac$-linear,
while the action of $\Dc$ may be encoded with an affine connection on the affine space.
In our chosen frame, differential operators $\dd{x_i} \in \Ds$ will act on
$\Ac \otimes U$ via $\dd{x_i} \otimes 1 + B_i$, where gauge fields $B_i$ are 
$\Ac$-linear operators satisfying the following two conditions:

(1) $[ B_i, B_j ] = \pd{B_i}{x_j} - \pd{B_j}{x_i}$,

(2) $\left[ \dd{x_i} \otimes 1 + B_i, \Lp \right] = 0$.

After introducing definitions and notations in Section 2, we prove our isomorphism theorem in Section 3. Then in Section 4 we show that $\AV$ modules of a finite type on an affine space are free. We apply these results in Section 5 to exhibit explicit structure 
of $\AV$ modules of a finite type on an affine space, proving a conjecture of
Billig-Futorny-Nilsson~\cite{BFN} in this case.

\end{section}

\begin{section}
{Definitions and Notations}

In this work we let $\kf$ be an algebraically closed field of characteristic zero. 
The algebra of functions on affine space $X=\Ab^n$ is $\Ac=\kf[x_1, \ldots, x_n]$. Lie algebra of polynomial vector fields on $\Ab^n$ is the Witt algebra
$\Vc=\Der \Ac=\bigoplus\limits_{i=1}^n \Ac\dfrac{\partial}{\partial x_i}$. 
This Lie algebra has a natural $\Zb$ grading by degree 
$\Vc=\Vc_{-1}\oplus \Vc_0\oplus \Vc_1 \oplus \cdots$.
We denote by $\Lp$ the subalgebra of vector fields that vanish at the origin: $\Lc_+ \simeq \Vc_0\oplus \Vc_1\oplus\cdots$.

Consider the Hopf algebra structure on $U(\Vc)$ with copruduct $\Delta$. For $u \in U(\Vc)$ we write $\Delta(u) = \sum\limits_i u_i^{(1)} \otimes u_i^{(2)}$. 
The commutative algebra $\Ac$ is naturally a module for $U(\Vc)$. The smash product $\smash$ is defined as the vector space $\Ac \otimes_\kf U(\Vc)$ with
the associative product
$$ (f \otimes u) (g \otimes v) = \sum\limits_i (f u_i^{(1)} (g)) \otimes (u_i^{(2)} v), \ \ \ f, g \in \Ac, u, v \in U(\Vc).$$
In particular, for $f \in \Ac$ and $\eta \in \Vc$, we have
$$ (1 \otimes \eta) (f \otimes 1) = f \otimes \eta + \eta(f) \otimes 1 .$$
since $\Delta(\eta) = \eta \otimes 1 + 1 \otimes \eta.$


The algebra of differential operators $\Dc$ on affine space $\Ab^n$ (also known as the Weyl algebra) may be defined as an associative subalgebra in 
$\End_\kf (\Ac)$, generated by operators of multiplication by $x_i$ and differentiation $\dfrac{\partial}{\partial x_j}$, $1 \leq i,j \leq n$.
We will be using multi-index notations, denoting for $s = (s_1, \ldots, s_n) \in \Zb_{\geq 0}^n$:
$$ x^s = x_1^{s_1} \cdots  x_n^{s_n},$$
$$ \partial^s = \left(\dfrac{\partial}{\partial x_1}\right)^{s_1} \cdots  \left(\dfrac{\partial}{\partial x_n}\right)^{s_n}.$$
With these notations a basis of $\Dc$ is given by $\{ x^r \partial^s \, | \, r,s\in\Zb_{\geq 0}^n \}$.

\end{section}

\begin{section}
{An isomorphism of associative algebras}

In this section we will prove the following theorem:

\begin{theorem}
\label{iso}
Let $\Ac = \kf[x_1, \ldots, x_n]$ be the algebra of functions on an affine space $\Ab^n$, 
let $\Vc = \mathop\oplus\limits_{i=1}^n \Ac \dd{x_i}$ be the Lie algebra of vector fields on $\Ab^n$,
and let $\Dc$ be the algebra of differential operators on $\Ab^n$, i.e., the Weyl algebra.
The following associative algebras are isomorphic:
$$\Ac\# \Uc(\Vc) \cong \Ds\otimes \Uc(\Lc_+).$$
The isomorphisms
\begin{equation*}
\label{phimap}
\varphi:\, \smash\rightarrow \Ds\otimes \Uc(\Lc_+)
\end{equation*}
and its inverse
\begin{equation*}
\label{psimap}
\psi:\, \Ds\otimes \Uc(\Lc_+) \rightarrow \smash
\end{equation*}
are defined as
\begin{equation*}
\varphi_{|_{\Vc}}\left(x^k\frac{\partial}{\partial x_p}\right)=x^k\frac{\partial}{\partial x_p}\otimes 1+\left(\sum_{0<m\leq k}{k\choose m}x^{k-m}\otimes x^m\frac{\partial}{\partial x_p}\right),
\end{equation*}
$\varphi$ restricted to $\Ac$ is the natural embedding into $\Dc$,
\begin{align*}
&\psi_{|_\Ds}\left(x^r\partial^s\right)=x^r\#\left(\frac{\partial}{\partial x_1}\right)^{s_1} \cdots \left(\frac{\partial}{\partial x_n}\right)^{s_n}\\
&\psi_{|_{\Lc_+}}\left(x^m\frac{\partial}{\partial x_p}\right)=\sum_{0\leq k\leq m}\left((-1)^{m-k}{m \choose k} x^{m-k}\#x^k\frac{\partial}{\partial x_p}\right).
\end{align*}
\end{theorem}


The proof will be split into several lemmas, with the first auxillary lemma being purely combinatorial. 
 In this section $j$, $k$, $l$, $m$ will denote multi-indices and $x$, $y$ 
will denote multi-variables. We will denote the standard basis of $\Zb^n$ by
$\{ \eps_1, \ldots, \eps_n \}$.

\begin{lemma}
\label{comb}
(a) 
\begin{equation*}
\sum_{0 \leq m \leq k} (-1)^m {k \choose m} = \delta_{k, 0},
\end{equation*}
(b)
\begin{align*}
&\sum_{0 < m \leq k} \sum_{0 < j \leq l} {k \choose m} {l \choose j}
j_p x^{k+l-m-j} y^{m+j-\eps_p} = \\
&l_p \sum_{0 < j \leq k +l - \eps_p} {k + l - \eps_p \choose j} 
x^{k+l-j-\eps_p} y^j
- l_p \sum_{0 < j \leq l - \eps_p} {l - \eps_p \choose j} 
x^{k+l-j-\eps_p} y^j.
\end{align*}
(c)
\begin{align*}
\sum_{0 \leq m \leq k} \sum_{0 \leq j \leq l} &(-1)^{k+l-m-j} {k \choose m} {l \choose j}
j_p x^{k+l-m-j} y^{m+j-\eps_p} = \\
&l_p \sum_{0 \leq j \leq k +l - \eps_p} (-1)^{k+l-j-\eps_p} {k + l - \eps_p \choose j} 
x^{k+l-j-\eps_p} y^j
\end{align*}

\end{lemma}
\begin{proof}
Part (a) is an immediate consequence of the binomial formula. Let us prove the formula in (b). Left hand side simplifies to 
\begin{align*}
\left( (x+y)^k - x^k \right) \dd{y_p} \left( (x+y)^l - x^l \right) 
= l_p (x+y)^{k+l-\eps_p} - l_p x^k (x+y)^{l_p - \eps_p} \\
=  l_p \left( (x+y)^{k+l-\eps_p} - x^{k+l-\eps_p} \right)
 - l_p x^k \left( (x+y)^{l_p - \eps_p} - x^{l_p - \eps_p} \right),
\end{align*} 
and we obtained the right hand side. The proof of part (c) is completely analogous.
\end{proof}

\begin{lemma}
\label{phi}
The map $\varphi:\, \smash\rightarrow \Ds\otimes \Uc(\Lc_+)$ 
given in Theorem~\ref{iso} is a homomorphism of associative algebras.
\end{lemma}
\begin{proof}
Clearly, the natural embedding $\Ac \hookrightarrow \Dc$ is a homomorphism.
To show that $\varphi_{|_{\Vc}}$ extends to a homomorphism of $\Uc (\Vc)$, we
need to check that
\begin{equation}\label{L1}
\left[\varphi\left(x^k\dx{p}\right),\varphi\left(x^l\dx{q}\right)\right]=\varphi\left(\left[x^k\dx{p},x^l\dx{q}\right]\right).
\end{equation}
The left hand side becomes
 \begin{align*}
&\left[x^k\frac{\partial}{\partial x_p},x^l\frac{\partial}{\partial x_q}\right]\otimes 1
+\sum_{0<j\leq l}{l\choose j}\left[x^k\frac{\partial}{\partial x_p},x^{l-j}\right]\otimes x^j\frac{\partial}{\partial x_q}\\
&-\sum_{0<m\leq k}{k\choose m}\left[x^l\frac{\partial}{\partial x_q},
x^{k-m}\right]\otimes x^m\frac{\partial}{\partial x_p}\\
+&\sum_{\substack{0<m\leq k\\0<j\leq l}}{k\choose m}{l\choose j}x^{k+l-m-j}\otimes\left[x^m\frac{\partial}{\partial x_p},x^j\frac{\partial}{\partial x_q}\right].
\end{align*}
We can see that 
$\left[x^k\frac{\partial}{\partial x_p},x^l\frac{\partial}{\partial x_q}\right]\otimes 1$
matches the first term of the right hand side of (\ref{L1}).  
Just as the right hand side of  (\ref{L1}), the remaining summands have terms with 
$\dd{x_q}$ and $\dd{x_p}$. Because of symmetry, it is sufficient to verify the equality 
for the terms with $\dd{x_q}$:
\begin{align*}
&\sum_{0<j\leq l}{l\choose j} (l_p - j_p) x^{k+l-j-\eps_p}\otimes x^j\frac{\partial}{\partial x_q}\\
+&\sum_{\substack{0<m\leq k\\0<j\leq l}}{k\choose m}{l\choose j}x^{k+l-m-j}\otimes j_p x^{m+j-\eps_p} \frac{\partial}{\partial x_q}
\end{align*}
Note that
$${l \choose j} (l_p - j_p) = l_p {l - \eps_p \choose j},$$
while the last term may be simplified using Lemma~\ref{comb} (b) and we obtain
\begin{align*}
&l_p \sum_{0<j\leq l}{l - \eps_p \choose j} x^{k+l-j-\eps_p}\otimes x^j\frac{\partial}{\partial x_q}\\
+ &l_p \sum_{0 < j \leq k +l - \eps_p} {k + l - \eps_p \choose j} 
x^{k+l-j-\eps_p} \otimes x^j \dd{x_q} \\
- &l_p \sum_{0 < j \leq l - \eps_p} {l - \eps_p \choose j} 
x^{k+l-j-\eps_p} \otimes x^j \dd{x_q}.
\end{align*}
The first and the last terms cancel and we end up with exactly what we need, establishing 
(\ref{L1}).

Finally, we need to verify that the commutation relation between vector fields and functions, which defines the smash product, is preserved:
\begin{equation}
\label{L2}
\left[\varphi\left(x^k\dx{p}\right),\varphi\left(x^l\right)\right]=\varphi\left(\left[x^k\dx{p},x^l\right]\right).
\end{equation}
Indeed,
\begin{align*}
&\left[\varphi\left(x^k\dx{p}\right),\varphi\left(x^l\right)\right]=\\
&\left[ x^k\dx{p} \otimes 1  
+\sum_{0<m\leq k}{k\choose m}x^{k-m}\otimes x^m\frac{\partial}{\partial x_p},
x^l \otimes 1 \right] =\\
&\left[ x^k\dx{p}, x^l \right] \otimes 1 =
\varphi\left(\left[x^k\dx{p},x^l\right]\right).
\end{align*} 
This completes the proof of Lemma~\ref{phi}.
\end{proof}

\begin{lemma}
\label{psi}
The map $\psi:\, \Ds\otimes \Uc(\Lc_+) \rightarrow \smash$ 
given in Theorem~\ref{iso} is a homomorphism of associative algebras.
\end{lemma}
\begin{proof}
We need to show that $\psi$ preserves the defining relations of $U(\Lc_+)$, $\Ds$ and
commutativity between $U(\Lc_+)$ and $\Ds$. Let us check that
\begin{equation}\label{L3}
\left[\psi\left(x^k\dx{p}\right),\psi\left(x^l\dx{q}\right)\right]=\psi\left(\left[x^k\dx{p},x^l\dx{q}\right]\right),
\end{equation}
where $k, l \neq 0$.
The left hand side becomes:
\begin{align*}
&\left[ \sum_{0\leq m \leq k} (-1)^{k-m} {k \choose m} x^{k-m} \# x^m \dd{x_p},
\sum_{0\leq j \leq l} (-1)^{l-j} {l \choose j} x^{l-j} \# x^j \dd{x_q} \right] =\\
&\sum_{\substack{0\leq m\leq k\\0\leq j\leq l}} (-1)^{k+l-m-j} {k \choose m} {l \choose j}
x^{k-m} \left[ x^m \dd{x_p}, x^{l-j} \right] \# x^j \dd{x_q} \\
&- \sum_{\substack{0\leq m\leq k\\0\leq j\leq l}} (-1)^{k+l-m-j} {k \choose m} {l \choose j}
x^{l-j} \left[ x^j \dd{x_q}, x^{k-m} \right] \# x^m \dd{x_p} \\
&+ \sum_{\substack{0\leq m\leq k\\0 \leq j\leq l}} (-1)^{k+l-m-j} {k \choose m} {l \choose j}
x^{k+l-m-j} \# \left[ x^m \dd{x_p}, x^j \dd{x_q} \right] .
\end{align*}
Just as we did in the proof of Lemma~\ref{phi}, let us keep track of the terms with $\dd{x_q}$, the terms with $\dd{x_p}$ can be then obtained by symmetry. Such terms will be
\begin{align*}
&\sum_{\substack{0\leq m\leq k\\0\leq j\leq l}} (-1)^{k+l-m-j} {k \choose m} {l \choose j}
(l_p - j_p) x^{k+l-j-\eps_p} \# x^j \dd{x_q} \\
&+ \sum_{\substack{0\leq m\leq k\\0\leq j\leq l}} (-1)^{k+l-m-j} {k \choose m} {l \choose j}
x^{k+l-m-j} \# j_p x^{m+j - \eps_p} \dd{x_q}  .
\end{align*}
In the first term above we can carry out summation in $m$, and we see that this term 
vanishes by Lemma~\ref{comb} (a), since $k \neq 0$. The second term may be evaluated 
using Lemma~\ref{comb} (c), yielding
\begin{align*}
&l_p \sum_{0 \leq j \leq k +l - \eps_p} (-1)^{k+l-j-\eps_p} {k + l - \eps_p \choose j} 
x^{k+l-j-\eps_p} \# x^j \dd{x_q} \\
& = l_p \psi \left( x^{k+l-\eps_p} \dd{x_p} \right),
\end{align*}
which is exactly what we need.

Next, let us verify that $\psi$ preserves the relations in $\Ds$. The defining relations of the Weyl algebra are: $[x_p , x_q] = 0$, $\left[ \dd{x_p}, \dd{x_q} \right] = 0$,
$\left[ \dd{x_p}, x_q \right] = \delta_{p,q}$. The fact that $\psi$ preserves the first two relations is a trivial consequence of the definition of $\psi$. The last relation follows 
immediately from the properties of the smash product.

 Finally, we need to check that $\psi(\Ds)$ and $\psi(\Uc (\Lc_+))$ commute. It is sufficient to establish commutativity for the generators of these associative algebras.

Consider
\begin{align*}
&\left[ \psi \left( \dd{x_q} \otimes 1 \right), \psi \left( 1 \otimes x^k \dd{x_p} \right) \right] =\\
&\left[ 1 \# \dd{x_q}, \sum_{0 \leq m \leq k} (-1)^m {k \choose m} x^{k-m} \#
x^m \dd{x_p} \right] = \\
&\sum_{0 \leq m \leq k} (-1)^m {k \choose m} \left[ \dd{x_q}, x^{k-m} \right]
\# x^m \dd{x_p} \\
&+  \sum_{0 \leq m \leq k} (-1)^m {k \choose m} x^{k-m} \#  \left[ \dd{x_q}, 
x^m \dd{x_p} \right] = \\
&\sum_{0 \leq m \leq k} (-1)^m {k \choose m} (k_q - m_q) x^{k-m-\eps_q}
\# x^m \dd{x_p} \\
&+  \sum_{0 \leq m \leq k} (-1)^m {k \choose m} x^{k-m} \#  m_q  x^{m-\eps_q} \dd{x_p} =\\
&k_q  \sum_{0 \leq m \leq k} (-1)^m {k - \eps_q \choose m} x^{k-m-\eps_q}
\# x^m \dd{x_p} \\
&+  k_q \sum_{0 \leq m \leq k} (-1)^m {k - \eps_q \choose m - \eps_q} x^{k-m} \#  m_q  x^{m-\eps_q} \dd{x_p} .\\
\end{align*}
Shifting the summation index $m$ by $\eps_q$ in the last term, we see that the two terms cancel out and the result is zero.

For the generator $x_q$ of $\Ds$ we get
\begin{align*}
&\left[ \psi(x_q \otimes 1), \psi\left(1 \otimes x^k \dd{x_p} \right) \right] = \\
& \left[ x_q \# 1,  \sum_{0 \leq m \leq k} (-1)^m {k \choose m} x^{k-m} \#
x^m \dd{x_p} \right] = \\
& -  \sum_{0 \leq m \leq k} (-1)^m {k \choose m} x^{k-m} 
\left[ x^m \dd{x_p}, x_q \right] \# 1 = \\
&- \delta_{p,q} \sum_{0 \leq m \leq k} (-1)^m {k \choose m} x^k \# 1 = 0,
\end{align*}
since $k \neq 0$. This completes the proof of Lemma~\ref{psi}.
\end{proof}

\begin{lemma}
\label{inv}
The homomorphisms $\varphi$ and $\psi$ given in Theorem~\ref{iso} are inverses of
each other.
\end{lemma}
\begin{proof}
We need to show that $\psi \circ \varphi$ and $\varphi \circ \psi$ are identity maps. 
Since these compositions are endomorphisms of associative algebras, it is sufficient to
check that each is an identity on the generators. 

The composition $\psi \circ \varphi$ is identity on $\Ac$ immediately from the definition.
Let us check that it is also identity on the elements of $\Vc$:
\begin{align*}
&\psi \left( \varphi \left( 1 \# x^k \dd{x_p} \right) \right) = \\
&\psi \left( x^k \dd{x_p} \otimes 1 \right) 
+ \sum_{0 < m \leq k} {k \choose m} \psi \left( x^{k-m} \otimes x^m \dd{x_p} \right) = \\
&x^k \# \dd{x_p} + 
\sum_{0 < m \leq k} {k \choose m}
\sum_{0 \leq j \leq m} (-1)^{m-j} {m \choose j} x^{k-m} x^{m-j} \# x^j \dd{x_p} = \\
&\sum_{0 \leq m \leq k} \sum_{0 \leq j \leq m}  (-1)^{m-j} {k \choose m}
{m \choose j} x^{k-j} \#  x^j \dd{x_p} = \\
& \sum_{0 \leq j \leq k} {k \choose j} \left( \sum_{j \leq m \leq k} (-1)^{m-j}
{k-j \choose m-j} \right)  x^{k-j} \#  x^j \dd{x_p} = \\
& \sum_{0 \leq j \leq k} {k \choose j} \delta_{j,k} x^{k-j} \#  x^j \dd{x_p} =
1 \# x^k \dd{x_p} .
\end{align*} 
The composition $\varphi \circ \psi$ is trivially identity on $\Ds$. Let us compute its 
value on $\Lc_+$:
\begin{align*}
&\varphi \left( \psi \left( 1 \otimes x^k \dd{x_p} \right) \right) = \\
&\sum_{0 \leq m \leq k} (-1)^{k-m} {k \choose m} \varphi \left( x^{k-m} \# x^m \dd{x_p} \right) = \\
&\sum_{0 \leq m \leq k} (-1)^{k-m} {k \choose m} (x^{k-m} \otimes 1)
\left( x^m \dd{x_p} \otimes 1 
+ \sum_{0 < j \leq m} {m \choose j} x^{m-j} \otimes x^j \dd{x_p} \right) = \\
&\sum_{0 \leq m \leq k} (-1)^{k-m} {k \choose m} x^k \dd{x_p} \otimes 1 \\
&+ \sum_{0 \leq m \leq k}  \sum_{0 < j \leq m} (-1)^{k-m} {k \choose m} {m \choose j}
x^{k-j} \otimes x^j \dd{x_p}.
\end{align*}
The first summand in the last equality vanishes, while in the second we switch the order
of summation:
\begin{align*}
&\sum_{0 < j \leq k} {k \choose j} \left( \sum_{j \leq m \leq k} (-1)^{k-m} {k-j \choose m-j} \right)
x^{k-j} \otimes x^j \dd{x_p} = \\
&\sum_{0 < j \leq k} {k \choose j} \delta_{j,k} x^{k-j} \otimes x^j \dd{x_p} = 
1 \otimes x^k \dd{x_p} .
\end{align*}
\end{proof}

Combining Lemmas~\ref{phi} -~\ref{inv}, we establish Theorem~\ref{iso}.

For any affine variety $X$ every $\Ds$ module is an $\AV$ module, but the converse is not necessarily true. However, Theorem~\ref{iso} shows that when $X$ is an affine space, every $\AV$ module has a $\Ds$ module structure, but with an additional commuting action of $\Lc_+$. It can be seen that in this case $\Ds$ modules are precisely
$\AV$ modules for which the action of $\Lc_+$ is zero.

\end{section}
\begin{section}{$\Ac\Vc$ modules that are finitely generated $\Ac$ modules}
  Let $\kf$ be an algebraically closed field of characteristic zero.  Let $X$ be an irreducible nonsingular affine algebraic variety over $\kf$.
  We let $\Ac =\Oc(X)$ be the algebra of regular functions on $X$, and $\Vc = \Der_\kf(A)$ the Lie algebra of vector fields on $X$.
  We recall the following two well known facts. 
  \begin{theorem} \label{wellknown} Let $\kf$ be an algebraically closed field of characteristic zero and let $X$ be a smooth irreducible variety over $\kf$ with $\Ac =\Oc(X)$
    and $\Vc = \Der_\kf(X)$, then
          \begin{enumerate}
  \item $\Dc(X)$ is generated by $\Ac$ and $\Vc$.
    \item There is a surjective algebra homomorphism $\Ac  \# U(\Vc) \rightarrow \Dc(X).$
    \item $\Ac$ is a simple $\Dc(X)$ module.
    \item \label{AisAVsimple} $\Ac$ is a simple $\Ac \# U(\Vc)$ module. 
  \end{enumerate}
  \end{theorem}
  \begin{proof}
    For (1) see Proposition 5.3 in~\cite{G}, and (2) follows from (1). To establish (3), we note that every $\Ds$-submodule in $\Ac$ is an ideal. Given a non-zero function in a
$\Ds$-submodule and a point $P \in X$, we can apply a sequence of derivations to obtain a function in the same $\Ds$-submodule, which does not vanish at $P$ (see e.g.~\cite{BF} for details). By Hilbert's Nullstellensatz, every non-zero $\Ds$-submodule in $\Ac$ coincides with $\Ac$. Finally, (4) follows from (3).
    \end{proof}
 For a left $\Ac$ module, recall the torsion submodule is defined by
 $$\Tor_{\Ac}(M) = \{ m \in M \, | \, \exists 
f \in \Ac, f \neq 0, ~\text{with}~ fm =0 \}.$$
 
  \begin{lemma}\label{Tor of Av mod}  Let $\kf$ be an algebraically closed field and let $X$ be an irreducible smooth affine variety over $\kf$.  
  Let $M$ be an $\Ac\Vc$ module that is finitely generated as an $\Ac$ module.
  Then $\Tor_{\Ac}(M)=0$.
\end{lemma}
\begin{proof}
  We first show that $\Tor_{\Ac}(M)$ is an $\Ac\Vc$ module. To this end let $m\in \Tor_{\Ac}(M)$ and let $f.m=0$ for some non zero divisor $f\in \Ac$. Then, for $\eta \in \Vc$ we have
  \begin{align*}
f^2.\eta.m&=-\eta(f^2).m+\eta.(f^2.m)\\
&=-2\eta(f)f.m+0\\
&=0
\end{align*}
and since $f^2$ is a non-zero divisor,  we see that $\eta.m \in \Tor_{\Ac}M$.
Since $M$ is finitely generated as an $\Ac$ module, then $\Tor_{\Ac}M$ is also finitely generated, say  $m_1, \ldots, m_n$. Then therefore there exist non-zero divisors $f_1, \ldots, f_n \in \Ac$ such that $f_im_i=0$, $1\leq i\leq n$. Now if $f=f_1f_2\cdots f_n$, then $f.\Tor_{\Ac}M=0$. Let $m'\in \Tor_{\Ac}M$ and $\eta\in \Vc$, then
\begin{align}\label{tozero}
\eta(f).m'&=\eta.(f.m')-f.(\eta.m')=0
\end{align}
Now by part~\ref{AisAVsimple} of Theorem~\ref{wellknown} we can find
$\sum g_i \otimes u_i \in \Ac \# U(\Vc)$ such that $(\sum g_i \otimes u_i)(f)=1$.  Using equation~(\ref{tozero}) repeatedly, we obtain $ 1.m'=0$. Hence $\Tor_{\Ac}M=0$.
\end{proof}


\begin{definition}
Let $A$ be a local ring with maximal ideal $\mf$ and let $M$ be a finitely generated  $A$ module. Then $\depth(M)$  is the supremum of the lengths of sequences $f_1, \ldots, f_r\in \mf$  such that $f_i$ is a nonzero divisor on $\frac{M}{f_1M+\cdots f_iM}$. The module $M$ is called maximal Cohen-Macaulay if $\depth(M)=\dim(A)$.  More generally, a module $M$ over a commutative ring $A$ is maximal Cohen-Macaulay if $M_\pf$ is maximal Cohen-Macaulay module over $A_\pf$ for all $\pf \in \Spec A.$ 
\end{definition}

\begin{proposition}\label{Av-CM}
Let $\Ac$ be  the algebra of polynomials over $x_1,\ldots,x_n$ and let  $\Vc=\Der \Ac=\bigoplus\limits_{i=1}^n \Ac\dfrac{\partial}{\partial x_i}$ be the Lie algebra of derivations of $\Ac$. If $M$ is an $\Ac\Vc$ module of finite type, then $M$ is a maximal Cohen-Macaulay $\Ac$ module.
\end{proposition}
\begin{proof} We proceed by induction on the dimension $n$.  Without loss of generallity, we let $\mf = (x_1,\ldots, x_n)$.
  By Lemma~\ref{Tor of Av mod} we see that   $x_1$ is a non-zero divisor of  $M$. Now consider the map $M \xrightarrow{x_1} M$ defined by multiplication by $x_1.$ Suppose that this map is surjective.  Since $M$ is torsion free, the map is injective.
  Let $x_1^{-1}:M\rightarrow M$ represent the inverse map. Then $M$ has the structure of a $\kf[x_1, \ldots,x_n,x_1^{-1}]$ module.  Let $m \in M$ with $m \neq 0$.  Then  $\kf[x_1, \ldots,x_n,x_1^{-1}]m$ is an $\Ac$-submodule of $M$ which is not finitely generated which violates the property of $M$ being Noetherian. So we can conclude that the map given by multiplication by $x_1$ is not surjective and is therefore a regular element on $M$.  Now consider the module
  $M' = \frac{M}{x_1M}.$  Note that $M'$ is an $\Ac\Vc$ module for $\Ac = \kf[x_2,\ldots,x_n] \simeq \kf[x_1,\ldots,x_n]/(x_1)$ and $\Vc = \Der(\Ac)$ since
  derivations on $\Ac$ map to derivations on $\kf[x_1,\ldots,x_n]$ that preserve $(x_1)$.  So by induction
  $\depth(M') = \dim(\Ac_\mf)=n-1$ and we are done.
\end{proof}



The following remark is restating~\cite[Corollary 19.6 \& Theorem 19.9]{E}.
\begin{proposition}\label{Auslander formula}
If $(A,\mf)$ is a regular local ring of finite dimension, then its global dimension is finite and therefore the projective dimension of any $A$ module is finite. Further, if $M$ is a finitely generated $A$ module, then
\[
pd~M=\dim(R) - depth(M)
\]
\end{proposition}

\begin{corollary}\label{CM-projective}
If $M$ is a finitely generated maximal Cohen-Macaulay module over a local regular ring $(A, \mf)$, then $M$ is projective.
\end{corollary}

\begin{proposition}\label{pro-free}
If $M$ is a finitely generated projective module over a ring of polynomials, then by~\cite[Theorem 4]{D}
 $M$ is free.
\end{proposition}

Considering Proposition~\ref{Av-CM} and Remark~\ref{pro-free}, we get the following corollary.
\begin{corollary}
\label{free}
Let $\Ac$ be  the algebra of polynomials over $x_1,\ldots,x_n$ and let  $\Vc=\Der \Ac=\bigoplus\limits_{i=1}^n \Ac\dfrac{\partial}{\partial x_i}$ be the Lie algebra of derivations of $\Ac$. If $M$ is an $\Ac\Vc$ module of finite type, then $M$ is a free $\Ac$ module
\end{corollary}
\begin{remark}
  It is a folklore statement that a coherent sheaf with a connection on a smooth variety $X$ is locally free.  Note that a coherent sheaf with connection is a $\Dc$ module, which is finitely generated over $\Ac$. Hence we obtain this folklore result in the case that $X$ is affine $n$-space.
\end{remark}
\end{section}


\begin{section}
{Gauge modules over affine space}

The goal of this section is to prove a conjecture stated in~\cite{BFN} in case when $X = \Ab^n$, showing that every $\AV$ module of a finite type is a gauge module.

The construction of a gauge module involves two ingredients: gauge fields and the action of the Lie algebra $\Lp$.

In what follows, $X = \Ab^n$, $\Ac = \kf[x_1, \ldots, x_n]$,  $\Vc$ is the Witt Lie algebra $W_n$ and $\Ds$ is the Weyl algebra of differential operators on $\Ab^n$.

Let $M$ be a free $\Ac$ module of a finite rank, $M = \Ac \otimes_\kf U$, where $U$
is a finite-dimensional vector space. Let $\rho$ be an $\Ac$-linear representation of the Lie algebra $\Lp$ on $\Ac \otimes U$.

Gauge fields are $\Ac$-linear maps $B_i : \AU \rightarrow \AU$, $i=1,\ldots,n$, 
satisfying
\leqnomode
\begin{align*}
\label{GF1}\tag{GF1}
&\left[ \dx{i} + B_i , \dx{j} + B_j \right] = 0, \\
\label{GF2}\tag{GF2}
&\left[ \dx{i} + B_i \ , \ \rho(\Lp)  \right] = 0. \\
\end{align*}
\reqnomode

\begin{definition}
A gauge module over affine space $\Ab^n$ is a free $\Ac$ module of a finite rank,
$M = \AU$, with an $\Ac$-linear action $\rho$ of Lie algebra $\Lp$ and gauge fields
$\{ B_i \}$, $i = 1, \ldots, n$, with the following action of Lie algebra $\Vc$:
\begin{equation}
\left( f \dx{i} \right) (g u) = 
f \frac{\partial g}{\partial x_i} u
+ fg B_i (u)
+ \sum\limits_{k \in \Zb^n_+ \backslash \{ 0 \}} \frac{1}{k!} g  \frac{\partial^k f}{\partial x^k}
\rho \left( x^k \frac{\partial}{\partial x_i} \right) u ,
\end{equation}
where $f, g \in \Ac$, $u\in  U$.
\end{definition}

The following lemma shows that the sum in the above formula is actually finite, so it is well-defined.

\begin{lemma}
Let $\rho$ be an $\Ac$-linear representation of $\Lp$ on $\Ac \otimes_\kf U$, where 
$\dim_\kf U < \infty$. Then there are only finitely many 
$k \in  \Zb^n_+ \backslash \{ 0 \}$ for which 
$\rho \left( x^k \frac{\partial}{\partial x_i} \right) \neq 0$.
\end{lemma}
\begin{proof}
Let $\bF$ be the algebraic closure of the field of fractions of $\Ac$. Then $\rho$ extends 
to a finite-dimensional representation $\bro$ on $\bF \otimes_{\kf} U$ of a Lie algebra
$\bF \otimes_\kf \Lp$ over the field $\bF$. By Lemma 3.4 from~\cite{B}, there are finitely many $k$ with $\bro \left( x^k \frac{\partial}{\partial x_i} \right) \neq 0$, and the claim of the lemma follows.
\end{proof}

\begin{remark}
Note that in the definition of a gauge module in~\cite{BFN}, $U$ was taken to be a finite-dimensional representation of $\Lp$. That definition should be generalized as we do here,
and we should rather require $\AU$ to be an $\Ac$-linear representaion of $\Lp$.
\end{remark}

\begin{theorem}
\label{aff-gauge}
Let $X = \Ab^n$. Every $\AV$ module of a finite type is a gauge module.
\end{theorem}
\begin{proof}
Let $M$ be an $\AV$ module of a finite type. By Corollary~\ref{free}, it is a free $\Ac$ module of a finite rank, and we can realize $M$ as $\Ac \otimes_\kf U$, where $U$ is
a finite-dimensional vector space. By Theorem~\ref{iso}, an $\AV$ module
structure on $M$ is equivlent to the existence of commuting actions of the Weyl algebra
$\Ds$ of differential operators on $\Ab^n$ and of the Lie algebra $\Lp$.

Moreover, $\Ac$ module structure on $M$ can be recovered from its $\Ds$ module structure via inclusion $\Ac \subset \Ds$. Since the action of $\Lp$ on $M = \AU$ 
commutes with the action of $\Ds$, it must be $\Ac$-linear.

It remains to show that a $\Ds$ module structure on $\AU$ can be encoded with the gauge fields $\{ B_i \}$. Let us denote the action of $\Ds$ on $M$ by $\theta$. Consider
the maps
$$ \theta \left( \dx{i} \right) : \, U \rightarrow \AU $$
and let $B_i$ be their $\Ac$-linear extensions to $M$:
$$B_i : \, \AU \rightarrow \AU .$$
Since 
$$\theta \left( \dx{i} \right) g u = \theta \left( \dx{i} \right) \theta(g)  u
= \frac{\partial g}{\partial x_i} u + g \theta \left( \dx{i} \right) u,$$
we conclude that 
\begin{equation}
\label{theta-x}
\theta \left( \dx{i} \right)  = \dx{i} \otimes \id_U + B_i 
\end{equation}
on $\AU$. We conclude that maps $\{ B_i \}$ satisfy the definition of gauge fields.

Conversely, a set of gauge fields will define operators $\theta \left( \dx{i} \right)$ on 
$\AU$ via (\ref{theta-x}), which will satisfy relations
$$ \left[ \theta \left( \dx{i} \right), x_j \right]  = \delta_{i,j},$$
Hence operators $ \theta \left( \dx{i} \right)$ and  $x_j$ together will generate the Weyl algebra $\Ds$ acting on $M = \AU$, which will commute with $\Lp$-action. This completes the proof of the theorem.
\end{proof}
\end{section}

\end{document}